\documentclass[review]{elsarticle}

\usepackage{lineno}
\modulolinenumbers[5]
\usepackage{graphicx}
\usepackage{latexsym,amssymb}
\usepackage{xcolor}
\usepackage{amsfonts}
\usepackage{amsmath}
\usepackage{amsthm}

\allowdisplaybreaks[4]
\newtheorem{coro}{Corollary}[section]
\newtheorem{lem}{Lemma}[section]
\newtheorem{rem}{Remark}[section]
\newtheorem{defi}{Definition}[section]
\newtheorem{theo}{Theorem}[section]
\newtheorem{example}{Example}[section]
\usepackage{txfonts}

\journal{Journal of \LaTeX\ Templates}

\begin{document}

\begin{frontmatter}

\title{Energy-preserving multi-symplectic Runge-Kutta methods for Hamiltonian wave equations
}

\author[address1]{Chuchu Chen\corref{mycorrespondingauthor}} 
\cortext[mycorrespondingauthor]{Corresponding author}
\ead{chenchuchu@lsec.cc.ac.cn}
\author[address1]{Jialin Hong}
\ead{hjl@lsec.cc.ac.cn}
\author[address2]{Chol Sim}
\ead{simchol@star-co.net.kp}
\author[address2]{Kwang Sonwu}

\address[address1]{LSEC, ICMSEC, Academy of Mathematics and Systems Science, CAS, Beijing 100190; School of Mathematical Sciences, University of Chinese Academy of Sciences, Beijing 100049}
\address[address2]{Institute of Mathematics, Academy of Sciences, Pyongyang, Democratic People's Republic of Korea}

\begin{abstract}
  It is well-known that a numerical method which is at the same time geometric structure-preserving and physical property-preserving cannot exist in general for Hamiltonian partial differential equations.
    In this paper, we present a novel class of parametric multi-symplectic Runge-Kutta methods for Hamiltonian wave equations, which can also conserve energy simultaneously in a weaker sense with a suitable parameter.
The existence of such a parameter, which enforces the energy-preserving property, is proved  under certain assumptions on the fixed step sizes and the fixed initial condition. 
 We compare the proposed method with the classical multi-symplectic Runge-Kutta method in numerical experiments, which shows the remarkable energy-preserving property of the proposed method
 and illustrate the validity of theoretical results.
\end{abstract}


\end{frontmatter}

\section{Introduction}

 Hamiltonian wave equation is the most important mathematical model in some scientific fields like quantum mechanics, plasma physics, etc. This equation is a typical example of Hamiltonian partial differential equations (PDEs)
of the form: 
\begin{equation}\label{equation:1.1}
M\partial _tz+K\partial _xz=\nabla _zS(z),\qquad z\in R^n,
\end{equation}
where $M,K\in R^{n\times n}$ are two skew-symmetric matrices,
$S:~R^n\rightarrow R$ is a given smooth function (at least twice
continuously differentiable), $\nabla_zS(z)$ is the classical
gradient on $R^n$, and $x,t$ denote the spatial and
temporal directions, respectively.

 For Hamiltonian PDEs \eqref{equation:1.1}, the two most prominent characteristics
 are multi-symplecticity, i.e.,
\begin{equation}\label{equation:1.2}
\partial_t \omega + \partial _x \kappa =0
\end{equation}
with $\omega = dz \wedge M dz, \kappa = dz \wedge K dz$,
and conservativeness, for example, the energy  conservation law (ECL):
\begin{equation}\label{equation:1.3}
\partial_t E(z) + \partial _x F(z) =0
\end{equation}
with energy density and energy flux being
$
E(z)=S(z)-\frac{1}{2}z^{T} K \partial _x z
$
and
$
F(z)=\frac{1}{2}z^{T} K \partial _t z,
$
respectively,
and  the momentum conservation law (MCL):
\begin{equation}\label{equation:1.4}
\partial_t I(z) + \partial _x G(z) =0,
\end{equation} 
with momentum density  and momentum flux being
$
I(z)=\frac{1}{2}z^{T} M \partial _x z,
$
and 
$
G(z)=S(z)-\frac{1}{2}z^{T} M \partial _t z,
$
respectively.

A well-known principle to design numerical methods is that numerical methods should preserve as much as possible the intrinsic properties of the underlying system. 
Generally, the numerical approximations of Hamiltonian system fall into two aspects: geometric structure-preserving numerical methods and physical property-preserving numerical methods. 
We start from the numerical study for Hamiltonian ordinary differential equations (ODEs), which has formed a well-developed subject through several decades efforts.
On the construction of symplectic numerical methods, there are several approaches, e.g., symplectic Runge-Kutta (RK) type methods, methods based on generating functions, variational integrators, etc., while on the construction of physical property-preserving numerical methods, different approaches, e.g., methods with projection type techniques, discrete gradient methods, Hamiltonian boundary value methods, etc., are proposed.
Except for the symplecticity-preservation, another important feature of symplectic numerical methods is that they can preserve exactly quadratic invariants (specially, quadratic Hamiltonian functions).  Generally, the famous Ge-Marsden theorem (\citep{zm})
shows the nonexistence of a constant time stepping algorithm which is at the same time symplectic and energy conserving.
The efforts towards this purpose on methods inheriting both features are made in a weaker sense. For example, \cite{kmo} proposes the adaptive time stepping symplectic-energy-momentum integrators with the symplecticity being viewed in the space-time sense. The constructive idea by introducing a parameter in each step which can be suitably tuned in a way to enforce energy conservation, is introduced in \cite{stw}, and is developed refinedly in \cite{bgi,bit,wxl}. More precisely, they introduce a family of one-step methods $y_1(\alpha)=\Phi_{h}(y_0,\alpha)$ depending on a real parameter $\alpha$ such that this family of methods are symplectic for any fixed choice of  $\alpha$, and that a special value of the parameter can be chosen depending on $y_0$ and $h$ with the conservation of  energy at the same time. 
For a system of Hamiltonian PDEs, the multi-symplecticity and ECL/MCL are the most relevant features characterizing its intrinsic property.
A natural question arises: Can one find a numerical method which combines the multi-symplectic structure and ECL/MCL at the same time in certain sense? 
It is believed that the difficulties in such a problem not only come from the balance of geometric structure and physical property, but also result from 
the numerical analysis of PDEs such as the interaction of time and space,  etc.

For the numerical study of Hamiltonian PDEs, 
the
multi-symplectic structure is investigated 
and then a lot of
reliable numerical methods (e.g. \cite{gcw, hls, hs, r, wh,hl,hll,hjl2}) preserving the multi-symplectic structure,
 for instance, 
muti-symplectic RK/PRK methods, collocation methods, splitting methods, spectral methods, etc., 
have been developed.   Especially, we refer to \cite{r,hjl,hjll,lz} and
references therein for the multi-symplectic methods of Hamiltonian wave
equations. 
 On the physical property-preserving aspect, 
as we mentioned that classical conservation laws such as ECL \eqref{equation:1.3} and MCL \eqref{equation:1.4} play an important role in Hamiltonian PDEs. 
Though they locally character the conservativeness, they are equivalent to the global conservation laws when an appropriate boundary condition is endowed.
The conservation of quadratic ECL and MCL under multi-symplectic  methods is proved in \cite{r,hjl}.
The accuracy of conservation laws of energy and momentum for Hamiltonian PDEs under RK discretizations is investigated in \cite{hjl}.
The approximate preservation of the global energy, momentum, and all harmonic actions over long time under temporal symplectic methods and spatial spectral methods applied to semilinear wave equation is rigorously proved in \cite{chl}. 
There have been several works on numerically preserving local ECL, MCL of Hamiltonian PDEs, e.g., \cite{gcw} for a systematic framework.
However, as far as we know, there are no known results about numerical methods which preserve multi-symplectic structure and ECL/MCL simultaneously for Hamiltonian PDEs. 

The main aim of this paper is to propose multi-symplectic
 RK methods which share the property of energy conservation for Hamiltonian wave equation.
We apply parametrized symplectic RK methods to Hamiltonian PDEs in space and in time, respectively, with the same real parameter $\alpha$, which is proved to be a concatenated $\alpha$-RK method preserving the multi-symplectic structure for all the real parameters. The preservation of ECL under the concatenated $\alpha$-RK method is obtained by suitably tuning the parameter, that is, 
we can show that the parameter $\alpha^{*}$ exists at each element domain composed by  spatial and temporal step sizes, which leads to the preservation of multi-symplecticity and ECL. That is a weaker version of the standard conservativeness, since
the existence of this parameter depends on the step sizes $\Delta x,\Delta t$, and on the initial data.

This paper is organized as follows. In section \ref{sec:2}, we present Hamiltonian wave equation and its multi-symplectic form.
Then a family of RK methods concatenated in the spatial and temporal  directions is introduced.
In section  \ref{sec:3}, we propose a concatenated parametric RK method by using W-transformation, which is multi-symplectic for all parameters. A numerical method with multi-symplecticity and energy-preserving can be obtained by selecting a suitable parameter.
Then, in section  \ref{sec:4}, we prove the existence of such parameter  with the aid of
the Lyapunov-Schmidt decomposition method, the homotopy continuation method  and the implicit function theorem under some assumptions on
RK matrices, the spatial and temporal  step sizes, etc.
In section  \ref{sec:5}, we present some numerical experiments to show the effectiveness in energy-preserving of the presented method.
 A short conclusion is given in section  \ref{sec:6}.

\section{Multi-symplectic Runge-Kutta methods}
\label{sec:2}

Consider the scalar wave equation
 \begin{equation}\label{equation:2.1}
\partial _{tt}u=\partial_{xx}u-V^\prime (u),
\qquad (x,t)\in \Omega \subset R^2,
\end{equation}
where $V:R\rightarrow R$ is a smooth function,  which is a typical example of Hamiltonian PDE \eqref{equation:1.1}.
By introducing canonical momenta $v:=\partial_tu$, $w:=\partial_xu$ and defining the state variable
$
          z=(u,v,w)^{T}\in R^3,
$
we rewrite wave equation \eqref{equation:2.1} as
\begin{align}
\partial _tu&=v,  \nonumber\\
\partial _xu&=w, \label{equation:2.2} \\
\partial _tv-\partial _x w&=-V^\prime (u).\nonumber
\end{align}
Using this variable, we obtain
$$
   M =
      \left(\begin{array}{ccc}

              0   &-1 &0  \\
              1   &0  &0  \\
              0   &0  &0
            \end{array}   \right)
      \quad {\rm and} \qquad
   K =
      \left(\begin{array}{ccc}
              0   &0  &1\\
              0   &0  &0\\
             -1   &0  &0

            \end{array}   \right)\\
$$
as well as the Hamiltonian
$
    S(z)=\frac{1}{2}(v^{2}-w^{2})+V(u).
$

Multi-symplectic conservation law (\ref{equation:1.2}), for the wave
equation (\ref{equation:2.1}), is equivalent to
$
\partial _t [du \wedge dv]- \partial _x [du \wedge dw] = 0.
$ We also obtain the  ECL (\ref{equation:1.3}) with
\begin{equation}\label{equation:2.6}
E(z)=\frac{1}{2}(w^{2}+v^{2})+V(u) \quad {\rm and}  \quad  F(z)=-vw
\end{equation}
and the  MCL (\ref{equation:1.4}) with
$
I(z)=-vw$ and $G(z)=\frac{1}{2}(w^{2}+v^{2})-V(u),
$ 
 respectively.

Now, we start our study with a multi-symplectic  RK method for Hamiltonian wave equation (\ref{equation:2.1}). First, we recall the definition of multi-symplectic integrator for   Hamiltonian PDEs.
  For the purpose of numerical approximation, following \cite{wh}, we introduce a uniform grid $(x_j, t_k) \in R^{2}$, in the plan of $(x,t)$, with  a spatial step size
$\Delta x$ and a
temporal step size $\Delta t$. The approximated value of $z(x,t)$ at the
mesh point $(x_j, t_k)$ is denoted by $z_{j,k}$. A numerical
discretization of (\ref{equation:1.1}) and (\ref{equation:1.2}),
can be written, respectively, as
\begin{align}
M\partial^{j,k}_t z_{j,k} + K\partial^{j,k}_x z_{j,k}&= (\nabla_zS_{j,k})_{j,k},\label{equation:2.8}\\
\partial^{j,k}_t \omega_{j,k} + \partial^{j,k}_x \kappa_{j,k}&= 0,\label{equation:2.9}
\end{align}
where $S_{j,k}:=S(z_{j,k}, x_j, t_k)$, $\partial ^{j,k}_t$,
$\partial ^{j,k}_x$ are discretizations of the derivatives
$\partial_t$ and
$\partial_x$, respectively.
The numerical method
\eqref{equation:2.8} is called a multi-symplectic integrator of the
system \eqref{equation:1.1} if (\ref{equation:2.9}) is a discrete
conservation law of (\ref{equation:2.8}) (\cite{wh}).

Next, we consider multi-symplectic  RK methods to solve
Hamiltonian wave equations.  
Multi-symplectic formulation
\eqref{equation:2.2} of the nonlinear wave
equation \eqref{equation:2.1} is discretized in space and in time by
a pair of  RK methods with $s$, $r$ stages, respectively (\cite{r}).
This is called the concatenated Runge-Kutta method.
Then the resulting discretizations are as follows:
\begin{align}
&U_{i,m}=u^m_0+\Delta x \sum^s_{j=1}a_{ij}\partial_xU_{j,m},\nonumber\\
&W_{i,m}=w^m_0+\Delta x \sum^s_{j=1}a_{ij}\partial_xW_{j,m},\nonumber\\
&u^m_1=u^m_0+\Delta x \sum^s_{i=1}b_i\partial_xU_{i,m},\nonumber\\
& w^m_1=w^m_0+\Delta x \sum^s_{i=1}b_i\partial_xW_{i,m},\nonumber\\
&U_{i,m}=u^0_i+\Delta t \sum^r_{n=1}\tilde{a}_{mn}\partial_tU_{i,n},\nonumber\\
&V_{i,m}=v^0_i+\Delta t \sum^r_{n=1}\tilde{a}_{mn}\partial_tV_{i,n},\label{equation:2.10}\\
& u^1_i=u^0_i+\Delta t
\sum^r_{m=1}\tilde{b}_{m}\partial_tU_{i,m},\nonumber\\
& v^1_i=v^0_i+\Delta t
\sum^r_{m=1}\tilde{b}_{m}\partial_tV_{i,m},\nonumber\\
&\partial_t U_{i,m}=V_{i,m},\nonumber\\
&\partial_x U_{i,m}=W_{i,m},\nonumber\\
&\partial_t V_{i,m}-\partial_x W_{i,m}=-V'(U_{i,m}),\nonumber
\end{align}
where we introduce the notations
$ U_{i,m}\approx u(c_i\Delta x,d_m\Delta t), ~ u_i^1\approx u(c_i\Delta x,\Delta t),~
u^m_1\approx u(\Delta x,d_m\Delta t),~ \text{etc.,}
$
with
$
 c_i = \sum_{j=1}^{s} a_{ij}, ~ d_m = \sum_{n=1}^r \tilde a_{mn}.
$

\begin{rem}
It was proved that if the coefficients of the method (\ref{equation:2.10}) satisfy the following condition:
\begin{equation}\label{equaion:2.11}
\begin{cases}
 b_ia_{ij}+b_ja_{ji}-b_ib_j=0,\quad \forall~ i,j=1,2,\ldots,s,\\
 \tilde b_m \tilde a_{mn} + \tilde b_n \tilde a_{nm}- \tilde b_m
 \tilde b_n=0, \quad \forall~ m,n=1,2,\ldots,r,
\end{cases}
\end{equation}
or equivalently,
\begin{equation}
\begin{cases}
M\equiv BA+A^{T}B-bb^{T}=0,\tag{$\ref{equaion:2.11}^{\prime}$}\\
\tilde{M}\equiv\tilde{B}\tilde{A}+\tilde{A}^{T}\tilde{B}-\tilde{b}\tilde{b}^{T}=0,
\end{cases}
\end{equation}
where $B={\rm diag}(b)$ and $\tilde{B}={\rm diag}(\tilde{b})$,
then this method is multi-symplectic (\cite{r,wh}).  
\end{rem}

We give some
examples of the multi-symplectic  RK methods by using  the Butcher tableau.

\begin{example} If $r=s=1$ with the following Butcher tableaux,  we obtain a multi-symplectic Gauss
collocation method with midpoint in time and in space
respectively, i.e., the centered Preissman scheme.  
\begin{center}
  \begin{tabular}{c|ccc|cc}
  $\frac{1}{2}$& $\frac{1}{2}$ &\qquad\qquad & $\frac{1}{2}$ & $\frac{1}{2}$ & \\
    \cline{1-2}\cline{4-5}
    & $1$  &\qquad\qquad &  & $\frac{1}{2}$     \\
  \end{tabular}.
\end{center}
If $r=1$, $s=2$ with the following Butcher tableaux, we obtain a multi-symplectic Gauss collocation
method with midpoint in time and fourth order Gauss collocation method in space.
\begin{center}
  \begin{tabular}{c|ccc|ccc}
  $\frac{1}{2}$& $\frac{1}{2}$ &\qquad\qquad & $\frac{1}{2}-\frac{\sqrt{3}}{6}$ & $\frac{1}{4}$ & $\frac{1}{4}-\frac{\sqrt{3}}{6}$ &\\
    \cline{1-2}
    & $1$  &\qquad\qquad & $\frac{1}{2}+\frac{\sqrt{3}}{6}$ & $\frac{1}{4}+\frac{\sqrt{3}}{6}$ & $\frac{1}{4}$     \\
    \cline{4-6}
    \multicolumn{2}{c}{\quad} & &  & $\frac{1}{2}$  & $\frac{1}{2}$    &
  \end{tabular}.
\end{center}
\end{example}

For the numerical method (\ref{equation:2.10}), the discrete ECL corresponding to (\ref{equation:1.3}) with
(\ref{equation:2.6}) is
\begin{equation}\label{equaion:2.12}
\sum_{i=1}^s b_i[E_i^1-E_i^0] \Delta x + \sum_{m=1}^r \tilde b_m
[F_1^m - F_0^m] \Delta t = 0,
\end{equation}
with
\begin{equation}\label{equaion:2.13}
E_i^1=\frac{1}{2}((w_i^1)^2+(v_i^1)^2)+V(u_i^1), \qquad
F_1^m=-v_1^mw_1^m,
\end{equation}
 and the discrete MCL can be given in the same manner.
 It is well-known that if the method (\ref{equation:2.10}) is multi-symplectic, then this
 method preserves the discrete ECL \eqref{equaion:2.12} for quadratic $V$. The motivation of this paper is to construct a new numerical method which can preserve the multi-symplectic structure and the discrete  ECL for the general $V$ simultaneously.
 
 \section{Multi-symplectic $\alpha$-Runge-Kutta methods}
\label{sec:3}

The W-transformation is very useful in the characterization and
construction of A-stable  RK methods and is also practicable
to construct high order symplectic  RK type
methods. Now we give some definitions and
results on the W-transformation. 

Consider shifted and normalized
Legendre polynomials
$
P_k(x)=\frac{\sqrt{2k+1}}{k!}\frac{d^k}{dx^k}(x^k(x-1)^k)
$
in [0,1].
These polynomials satisfy the integration formulas:
\begin{align*}
&\int^{x}_{0}P_0(t)dt= \xi_{1}P_{1}(x)+\frac{1}{2}P_0(x),\\
&\int^{x}_{0}P_k(t)dt =\xi_{k+1}P_{k+1}(x)-\xi_{k}P_{k-1}(x),\qquad
k=1,2,\ldots
\end{align*}
with
$
\xi_{k}=\frac{1}{2\sqrt{4k^2-1}}.
$

The definition of a W-transformation relies on a generalized
Vandermonde matrix $W=(w_{ij})^s_{i,j=1}$ 
whose elements are the shifted and normalized
Legendre polynomials of degree $j$ for $j=0,1,\ldots
,s-1$, evaluated at $c_i~(i=1,2,\ldots  ,s)$, i.e., $w_{ij}=P_{j-1}(c_i).
$

\begin{defi}
\cite[p.81]{hw}. Let $\eta, \xi$ be given integers between
$0$ and $s-1$. We say that an $s\times s\text{-}$matrix $W$
satisfies $T(\eta,\xi)$ for the quadrature formula
$(b_i,c_i)^s_{i=1}$ if
\begin{itemize}
\item [(1)] $W$ is nonsingular;
\item [(2)] $w_{i,j}=P_{j-i}(c_i), ~i=1,2,\ldots,s,~j=1,2,\ldots,  \max(\eta,\xi)+1$;
\item [(3)] $W^TBW=\left(\begin{array}{cc}
              I &0\\
              0 &R
            \end{array}   \right),$
where $I$ is the $(\xi+1)\times (\xi+1)$ identity matrix and $R$ is
an arbitrary $(s-\xi-1)\times (s-\xi-1)$ matrix.
            \end{itemize}
   \end{defi}

If $W$ satisfies $T(\eta,\xi)$ for the quadrature formula $(b_i,c_i)_{i=1}^s$, then the W-transformation for an $s$-stage RK method is
defined by
\begin{equation}\label{equation:3.1}
X=W^TBAW,
\end{equation}
where $B={\rm diag}(b)$, and  $A$ is the coefficient matrix  of this RK method. Further, if $A$ is the coefficient matrix for Gauss method of order $2s$,  then
$$
X=W^TBAW=W^{-1}AW= \left(\begin{array}{cccccc}
     \frac{1}{2}   &-\xi_1   \\
     \xi_1   &0   &-\xi_2\\
        &\xi_2   &\ddots &\ddots\\
            &&\ddots &0 &-\xi_{s-1}\\
            &&&\xi_{s-1}  &0
     \end {array}\right),
$$
where $
\xi_{k}=\frac{1}{2\sqrt{4k^2-1}}
$ (see Theorem 5.6 and Theorem 5.9 in \cite{hw}). 
Note that the multi-symplectic condition $\eqref{equaion:2.11}^{\prime}$ can be written in the form:
\begin{equation}\label{multi-condition}
\begin{cases}
W^{T}MW=X+X^{T}-e_1e_1^{T}=0,\\
\tilde{W}^{T}\tilde{M}\tilde{W}=\tilde{X}+\tilde{X}^{T}-\tilde{e}_1\tilde{e}_1^{T}=0,
\end{cases}
\end{equation}
where $e_1=(1,0,\ldots,0)\in R^s$ and $\tilde{e}_1=(1,0,\ldots,0)^{T}\in R^r$.

Now, for two given $RK$ methods $(c,A,b)$ and $(\tilde c, \tilde A,
\tilde b)$ with the transformation matrices $X$ and $\tilde X$
defined by (\ref{equation:3.1}), we define the perturbed matrices
$X(\alpha)$ and $\tilde X(\alpha)$ as
\begin{equation}\label{equation:3.5}
X(\alpha)=X+\alpha V, \quad \tilde X(\alpha)=\tilde X+\alpha \tilde
V,
\end{equation}
where $\alpha$ is a real parameter and
\begin{equation*}
V= \left(\begin{array}{ccccccccc}
     0 &\cdots  &0   \\
     \vdots &\ddots  &&\ddots\\
     0   &&\ddots &&-1\\
     &\ddots  &&\ddots &&\ddots\\
     &&1 &&\ddots &&0\\
     &&&\ddots &&\ddots &\vdots\\
      &&&&0 &\cdots &0
     \end {array}\right).
\end{equation*}
Nonzero elements $1$ and $-1$ in the matrix $V$ should be arranged such that the matrix keeps to be skew-symmetric. The matrix $\tilde V$ is
defined similarly as $V$.

\begin{defi}
Given $s\text{-}$ and $r\text{-}$stage symplectic RK methods
$(c,A,b)$ and $(\tilde c, \tilde A, \tilde b)$, we define the
corresponding $\alpha\text{-}RK$ methods as $(c,A(\alpha),b)$,
$(\tilde c, \tilde A(\alpha), \tilde b)$, respectively, where
\begin{align}
A(\alpha)=(W^TB)^{-1}X(\alpha)W^{-1}=A+\alpha(W^TB)^{-1}VW^{-1},\label{equation:3.7}\\
\tilde A(\alpha)=(\tilde W^T \tilde B)^{-1}\tilde X(\alpha)\tilde
W^{-1}=\tilde A+\alpha(\tilde W^T\tilde B)^{-1}\tilde V \tilde
W^{-1}.\label{equation:3.8}
\end{align}
We call it a concatenated $\alpha\text{-}RK$ method, if the above $s\text{-}$ and $r\text{-}$stage $\alpha\text{-}RK$ methods
are applied to the system (\ref{equation:2.10}) in spatial direction and
temporal direction 
respectively.
\end{defi}

If the quadrature $(b,c)$ has order  $\geq 2s-1$ and $(\tilde b,\tilde c)$ has order  $\geq2r-1$, (\ref{equation:3.7}) and (\ref{equation:3.8}) are
reduced, respectively, to
\begin{align}
A(\alpha)=WX(\alpha)W^{-1}=A+\alpha WVW^{-1},\label{equation:3.9}\\
\tilde A(\alpha)=\tilde W\tilde X(\alpha)\tilde W^{-1}=\tilde A+\alpha \tilde W \tilde V \tilde W^{-1}.\label{equation:3.10}
\end{align}

\begin{theo}\label{Thm:3,5}
Given a multi-symplectic $RK$ method, if $W$ (resp. $\tilde{W}$) satisfies $T(\eta,\xi)$ for $(b_i,c_i)_{i=1}^s$ (resp. $(\tilde{b}_i,\tilde{c}_i)_{i=1}^{r}$), then  for any $\alpha$, the corresponding concatenated
$\alpha\text{-}RK$ method in 
(\ref{equation:3.7})-(\ref{equation:3.8}) is also multi-symplectic.
\end{theo}
\begin{proof}
By utilizing \eqref{multi-condition},  
the proof follows from the anti-symmetricity of $V$ and $\tilde{V}$.
\end{proof}

\begin{example}\label{example3.1}
Based on \eqref{equation:3.9}-\eqref{equation:3.10}, 
we  give an example of multi-symplectic $\alpha$-Gauss collocation method with $r=1$ and $s=2$, whose Butcher tableaux are as follows:
\begin{center}

  \begin{tabular}{c|ccc|ccc}

  $\frac{1}{2}$& $\frac{1}{2}$ &\qquad\qquad & $\frac{1}{2}-\frac{\sqrt{3}}{6}$ & $\frac{1}{4}$ & $\frac{1}{4}-\frac{\sqrt{3}}{6}-\alpha$ &\\
    \cline{1-2}
    & $1$  &\qquad\qquad & $\frac{1}{2}+\frac{\sqrt{3}}{6}$ & $\frac{1}{4}+\frac{\sqrt{3}}{6}+\alpha$ & $\frac{1}{4}$     \\
    \cline{4-6}
    \multicolumn{2}{c}{\quad} & &  & $\frac{1}{2}$ & $\frac{1}{2}$ &
  \end{tabular}.
\end{center}
Consequently, if $\alpha=0$, we retrieve the multi-symplectic Gauss
collocation method with $r=1$, $s=2$. 
\end{example}

 \begin{rem}
  Since matrix $V$  must be skew-symmetric, it  vanishes in the case of $r=1$. As shown in the above example, it is sufficient to introduce one
parameter $\alpha$ only in the spatial direction  for the preservation of energy. 
\end{rem}

Similarly to the case of $\alpha=0$, by  concatenating two
$\alpha$-Gauss collocation methods in 
spatial and temporal directions respectively, we can construct a lot of
multi-symplectic $\alpha$-Gauss collocation methods.

\begin{rem}
 If the periodic boundary
condition $u(0,t)=u(L,t)$ is endowed, \eqref{equation:1.3} yields the preservation of the  ``global" energy:
\begin{equation}\label{equation:3.12}
{\cal E}(t):=\int_{0}^L E(z(x,t))dx=\int_ {0}^L
E(z(x,0))dx=:{\cal E}(0) .
\end{equation}
One can also have the conservation of the  ``global" momentum if $u(0,t)=u(L,t)$, i.e.,\begin{equation}\label{equation:3.14}
{\cal I}(t):=\int_{0}^L I(z(x,t))dx=\int_ {0}^L
I(z(x,0))dx=:{\cal I}(0).
\end{equation}
We will show through numerical experiments  that the
method proposed in this paper conserves well the local and
``global" energy in section \ref{sec:5}.  
\end{rem}

\section{Existence of energy-preserving multi-symplectic Runge-Kutta methods}
\label{sec:4}

In this section, we give an existence result of the parameter $\alpha^{*}$ which ensures the
energy-preserving property of the proposed multi-symplectic methods. The existence of the parameter is  in a weaker sense, which means that this parameter depends on the step sizes $\Delta x,\Delta t$, and on the initial data.

To find such a parameter $\alpha^{*}$,  we need
to solve a nonlinear system with $4rs+1$ unknowns:
\begin{equation}\label{equation:3.11}
\begin{cases}
U_{i,m}=u^m_0+\Delta x\sum^s_{j=1}a_{ij}(\alpha)W_{j,m},\\
W_{i,m}=w^m_0+\Delta x\sum^s_{j=1}a_{ij}(\alpha)\partial_xW_{j,m},\\
U_{i,m}=u^0_i+\Delta t\sum^r_{n=1}\tilde{a}_{mn}(\alpha)V_{i,n},\\
V_{i,m}=v^0_i+\Delta t\sum^r_{n=1}\tilde{a}_{mn}(\alpha)\partial_t
        V_{i,n},\\
\partial_t V_{i,m}-\partial_x W_{i,m}=-V'(U_{i,m}),\\
\sum_{i=1}^s b_i[E_i^1-E_i^0] \Delta x + \sum_{m=1}^r \tilde b_m
        [F_1^m - F_0^m] \Delta t = 0.
\end{cases}
\end{equation}
Denoting  
$$
U=(U_1,U_2,\ldots,U_r)^T, \quad  V=(V_1,V_2,\ldots,V_r)^T, \quad W=(W_1,W_2,\ldots,W_r)^T,
$$
$$
\partial_t V=(\partial_t V_1,\partial_t V_2,\ldots,\partial_t V_r)^T,\qquad
\partial_x W=(\partial_x W_1,\partial_x W_2,\ldots,\partial_x W_r)^T
$$
with
$
U_m=(U_{1m},U_{2m},\ldots,U_{sm}) \in R^s$
  and  $V_m,W_m,\partial_t V_m, \partial_x W_m\in R^s$, $m=1,2,\ldots, r$ being similarly defined,
the nonlinear system (\ref{equation:3.11}) can be rewritten in a
compact form:
\begin{equation}\label{equation:4.1}
\begin{cases}
U=u_0\otimes e_s + \Delta x[I_r\otimes A(\alpha)]W,\\
W=w_0\otimes e_s +\Delta x[I_r\otimes A(\alpha)][\partial_tV+R(U)],\\
U=e_r\otimes u^0 + \Delta t[\tilde A(\alpha)\otimes I_s]V,\\
V=e_r\otimes v^0 + \Delta t[\tilde A(\alpha)\otimes I_s]\partial_t V,\\
b^T(E^1-E^0)\Delta x+\tilde b^T(F_1-F_0)\Delta t=0,
\end{cases}
\end{equation}
where $\otimes$ denotes the Kronecker product, $u^0=(u_1^0,u_2^0,\ldots,u_s^0)^T$, $u^1=(u_1^1,u_2^1,\ldots,u_s^1)^T$,
$u_0=(u_0^1,u_0^2,\ldots,u_0^r)^T$, $u_1=(u_1^1,u_1^2,\ldots,u_1^r)^T$, and $v^0, v^1, v_0, v_1, w^0, w^1, w_0, w_1$ have the similar notions.
In addition, 
$$
e_s=(1,1,\ldots,1)^T\in R^s, \qquad
e_r=(1,1,\ldots,1)^T\in R^r,
$$
$$
R(U)=(\tilde V_1^T,\tilde V_2^T,\ldots,\tilde V_r^T)^T,
\qquad \tilde V_m=(V'(U_{1m}),\ldots,V'(U_{sm}))^T,
$$
$$
A(\alpha)=(a_{ij}(\alpha))^s_{i,j=1}, \qquad \tilde
A(\alpha)=(\tilde a_{mn}(\alpha))^r_{m,n=1},
$$
$$
E^1=(E_1^1,E_2^1,\ldots,E_s^1)^T,\quad
E^0=(E_1^0,E_2^0,\ldots,E_s^0)^T,
$$
$$
F_1=(F_1^1,F_1^2,\ldots,F_1^r)^T,\quad
F_0=(F_0^1,F_0^2,\ldots,F_0^r)^T,
$$
$$
b=(b_1,b_2,\ldots,b_s)^T, \qquad \tilde b=(\tilde b_1,
\tilde b_2, \ldots,\tilde b_r)^T.
$$
Let
$
S(u^0)=\left(V(u_1^0),V(u_2^0),\ldots,V(u_s^0)\right)^T$
and $
S(u^1)=\left(V(u_1^1),V(u_2^1),\ldots,V(u_s^1)\right)^T,
$ 
then the last equation in (\ref{equation:4.1}) becomes 
\begin{equation}\label{equation:4.2}
\begin{split}
\bigg[&\frac{1}{2}b^T{\rm diag}(w^1)w^1+\frac{1}{2}b^T{\rm
diag}(v^1)v^1 +b^TS(u^1)-\frac{1}{2}b^T{\rm
diag}(w^0)w^0\\
&-\frac{1}{2}b^T{\rm diag}(v^0)v^0
-b^TS(u^0)\bigg ]\Delta x
+\Big[\tilde b^T {\rm diag}(v_0)w_0-\tilde b^T {\rm diag}(v_1)w_1\Big]\Delta
t=0.
\end{split}
\end{equation}
We rewrite \eqref{equation:4.1} as
\begin{equation}\label{equation:4.3}
\begin{cases}
L(\Delta x, \Delta t,\alpha)Y-\Delta x F(Y,\alpha)=Y_0,\\
T(u^1,v^1,w^1,v_1,w_1)-T(u^0,v^0,w^0,v_0,w_0)=0,
\end{cases}
\end{equation}
where 
$
Y=(U^T,V^T,W^T,\partial_t V^T)^T,
$
$
Y_0=((u_0\otimes e_s)^T, (w_0\otimes e_s)^T,(e_r\otimes
u^0)^T,(e_r\otimes v^0)^T)^T,
$ 
$
T(x,y,z,p,q)=\left [\frac{1}{2}b^T {\rm diag}(z) z +\frac{1}{2}b^T
{\rm diag}(y) y+b^TS(x)\right ] \Delta x-\tilde b^T {\rm
diag}(p)q\Delta t,
$ and
$$
L(\Delta x, \Delta t, \alpha)= \left(\begin{array}{cccc}
I_{rs} &-\Delta x(I_r\otimes A(\alpha)) &O_{rs}  &O_{rs}\\
O_{rs}  &O_{rs}  &I_{rs} &-\Delta x(I_r\otimes A(\alpha))\\
I_{rs}  &-\Delta t(\tilde A(\alpha)\otimes I_s) &O_{rs}  &O_{rs}\\
O_{rs} &I_{rs} &O_{rs} &-\Delta t(\tilde A(\alpha)\otimes I_s)
\end{array} \right),
$$
$$
F(Y, \alpha)= \left(\begin{array}{cccc} 0_{rs}\\
(I_r\otimes A(\alpha))R(U)\\
0_{rs}\\
0_{rs} \end{array} \right),
$$
with $O_{rs}$ being an $(rs\times rs)$-zero matrix and $0_{rs}$ being an
$rs$-zero vector.

From (\ref{equation:2.10}), we have
\begin{align*}
&w_1^m=w_0^m+\Delta x \sum_{i=1}^s b_i(\partial_t V_{i,m}+R(U_{i,m})),\\
& u_i^1=u_i^0+\Delta t\sum_{m=1}^r \tilde b_m V_{i,m},
\\
&v_i^1=v_i^0+\Delta t\sum_{m=1}^r \tilde b_m \partial_t V_{i,m}.
\end{align*}
Meanwhile, we introduce two auxiliary systems:
\begin{align*}
&V_{i,m}=v_0^m+\Delta x \sum _{j=1}^s a_{ij}(\alpha)\partial_xV_{j,m},\\
&v_1^m=v_0^m+\Delta x\sum_{i=1}^sb_i\partial_xV_{i,m},\\
&W_{i,m}=w_i^0+\Delta t\sum_{n=1}^r\tilde a_{mn}(\alpha)\partial_tW_{i,n},\\
&w_i^1=w_i^0+\Delta t\sum_{m=1}^r\tilde b_m\partial_t W_{i,m}=w_i^0+\Delta t\sum_{m=1}^r\tilde b_m\partial_x V_{i,m},
\end{align*}
where we use $\partial_xV_{i,m}=\partial_t W_{i,m}$ in the last equation. Thus
the equations for $u_i^1, v_i^1, w_i^1, v_1^m, w_1^m$ can be written as
\begin{equation}\label{equation:4.4}
\begin{split}
&u^1=u^0+\Delta t \tilde B_* V,\\
&v^1=v^0+\Delta t \tilde B_* \partial_t V,\\
&w^1=w^0+\Delta t \tilde B_*\partial_x V\\
&v_1=v_0+\Delta x B_* \partial_x V,\\
&w_1=w_0+\Delta x B_* (\partial_t V+R(U)),
\end{split}
\end{equation}
where $\tilde B_*=\tilde b^T\otimes I_s$ and $B_*=I_r\otimes b^T$.

Now, let $y_1=(u^1,v^1,w^1,v_1,w_1)^T$, $y_0=(u^0,v^0,w^0,v_0,w_0)^T$.
Define the error function in the discrete  ECL as
$
G(\alpha)=T(y_1)-T(y_0).
$
When needed, $G(\alpha)$ may be written as $G(\alpha, \Delta x,\Delta t)$ to emphasize 
the dependence on mesh sizes $\Delta x$ and $\Delta t$.
The numerical solution of the multi-symplectic $\alpha\text{-}RK$ method
(\ref{equation:3.7})-(\ref{equation:3.8}) defines a
corresponding mapping of the form
$
y_1=\Phi_{\Delta x, \Delta t}(y_0, \alpha).
$
The nonlinear system (\ref{equation:4.1}), in which $4rs+1$ unknowns $U, V, W,\partial_t V$ and $\alpha$ need to be solved in every
rectangular domains composed of length $\Delta x$ and width $\Delta t$, reads
\begin{equation}\label{equation:4.6}
\begin{cases}
L(\Delta x,\Delta t, \alpha)Y=Y_0+\Delta xF(Y, \alpha),\\
G(\alpha)=0.
\end{cases}
\end{equation}
The solvability of this system yields the existence of the energy-preserving method.
For convenience, we denote  $h:=\Delta x$, $\tau:=\Delta t$, and define  the vector function
$$
\varphi(h,\tau,y_1,\alpha)=\left(\begin{array}{cc}
y_1-\Phi_{h,\tau}(y_0,\alpha)\\
T(y_1)-T(y_0)
\end{array}\right),
$$
then the system (\ref{equation:4.6}) is equivalent to
$\varphi(h,\tau,y_1,\alpha)=0$.

From (\ref{equation:4.4}), we know that
$\varphi(0,0,y_1,\alpha)=0$ for every  $y_1$ and $\alpha$. The Jacobian of
$\varphi$ with respect to $(y_1, \alpha)$ is
\begin{equation}\label{equation:4.7}
\frac{\partial \varphi}{\partial (y_1,\alpha)}(h,\tau,y_1,\alpha)=
\left(\begin{array}{cc}
I  &-\frac{\partial \Phi_{h,\tau}}{\partial \alpha}(y_0,\alpha)\\
\nabla^T T(y_1) &0
\end{array}\right)
\end{equation}
with $I$ being the identity matrix of dimension $3s+2r$. By using
the formula on determinant of block matrices, we get
$$
{\rm det}\left (\frac{\partial \varphi}{\partial
(y_1,\alpha)}(h,\tau,y_1,\alpha)\right )= {\rm det}\left(\nabla^T
T(y_1)\cdot \frac{\partial \Phi_{h,\tau}}{\partial
\alpha}(y_0,\alpha)\right).
$$
Since
$$
T(y_1)=\left [\frac{1}{2}b^T {\rm diag}(w^1)w^1+\frac{1}{2}b^T {\rm
diag}(v^1)v^1+b^T S(u^1)\right] h-\tilde b^T {\rm diag}(v_1)w_1
\tau,
$$
it holds that
\begin{align*}
\nabla^T T(y_1)=\bigg(&b^T S'_{u^1}(u^1)h,\quad \frac{1}{2}b^T({\rm
diag}(v^1)v^1)'_{v^1}h, \quad\frac{1}{2}b^T({\rm diag}(w^1)w^1)'_{w^1}h,\\
&-\tilde b^T ({\rm diag}(v_1)w_1)'_{v_1}\tau, \quad-\tilde b^T({\rm
diag}(v_1)w_1)'_{w_1}\tau \bigg ).
\end{align*}
Therefore, when $h=0$ and $\tau=0$, $\nabla^T T(y_1)$ is the zero
vector for any $\alpha$, and thus the rank of the matrix
(\ref{equation:4.7}) is $3s+2r$.

Due to singularity of the matrix  \eqref{equation:4.7} when $h=0$ and $\tau=0$, the implicit function theorem
cannot be applied directly to prove the solvability of the nonlinear system (\ref{equation:4.6}).
In \cite{bit},  the existence of parameter $\alpha^*$ for $\alpha$-Gauss collocation symplectic
method is proved with the aid of the Lyapunov-Schmidt
decomposition method. This decomposition method restricts the nonlinear system to
both the complement of a null space and the range of the Jacobian.
Thus one gets two subsystems whose Jacobians are nonsingular and then the implicit function theorem is
applicable.

Before we give the existence of $\alpha^{*}$, we first state the
solvability of the first system in \eqref{equation:4.6}, which could be proved similarly as  
in \cite[Theorem 3.1]{hsy}. 

\begin{lem}\label{lemma4.1}
If $h\leqslant \tau^2$
and RK matrix $\tilde A$ is nonsingular, then for
$|\alpha|\leqslant \alpha_0$, $h\leqslant h_0$ and $\tau\leqslant
\tau_0$ with $\alpha_0$, $h_0$ and $\tau_0$ small enough, there
exists a solution $Y(\alpha)$ of the first system in
(\ref{equation:4.6}).
\end{lem}

Now we are in the position to give the existence of $\alpha^{*}$. The following assumptions similar as in \cite{bit} are made:
\begin{itemize}
\item[($S_1$)] $h\leqslant \tau^2$;

\item[($S_2$)] The function $G$ is analytical in a cube
$\mathcal{Q}=[-\alpha_0, \alpha_0]\times [-h_0,h_0]\times
[-\tau_0,\tau_0]$;

\item[($S_3$)] $G(0,h,\tau)=c_0\tau^d+\mathcal{O}(\tau^{d+1}),\quad c_0\neq
0$,

 $G(\alpha,h,\tau)=c(\alpha)\tau^{d-m} +\mathcal{O}
(\tau^{d+1-m}),\quad c(\alpha)\neq 0.$
\end{itemize}

\begin{lem}\label{lemma4.2}
Under the assumptions ($S_1$)-($S_3$), there exists a function
$\alpha^*=\alpha^*(h,\tau)$ defined in a rectangle $(-h_0,h_0)\times
(-\tau_0,\tau_0)$, such that
\begin{itemize}
\item[(i)] $G(\alpha^*(h,\tau),h,\tau)=0, \quad \text{for all} ~h \in (-h_0,h_0)
\text{ and }~ \tau \in (-\tau_0,\tau_0)$;
\item[(ii)] $\alpha^*(h,\tau)={\rm const} \cdot \tau^m+
\mathcal{O}(\tau^{m+1})$ for some integer $m$.
\end{itemize}

\end{lem}
\begin{proof} From  $(S_2)$ and $(S_3)$, the expansion of $G$ around $(0,0,0)$ is
\begin{equation}\label{equation:4.9}
G(\alpha,h,\tau)=\sum_{i=0}^{\infty}\sum_{j=0}^{\infty}\sum_{k=0}^{\infty}
\frac{1}{i!j!k!}\cdot
\frac{\partial^{i+j+k}G}{\partial\alpha^i\partial h^j\partial\tau^k}
(0,0,0)\alpha^i h^j\tau^k.
\end{equation}
By $(S_1)$, there exists a constant
$\beta~(0<\beta<1)$ such that $h=\beta \tau^2$. Substituting this
into (\ref{equation:4.9}) leads to
\begin{align*}
&G(\alpha,h,\tau)=G(\alpha,h(\tau),\tau)=\sum_{i=0}^{\infty}\sum_{j=0}^{\infty}\sum_{k=0}^{\infty}
\frac{\beta^j}{i!j!k!}\cdot
\frac{\partial^{i+j+k}G}{\partial\alpha^i\partial h^j\partial\tau^k}
(0,0,0)\alpha^i\tau^{2j+k}\\
&=\sum_{j=0}^{\infty}\sum_{k=0}^{\infty}\frac{\beta^j}{j!k!}\cdot
\frac{\partial^{j+k}G}{\partial h^j\partial\tau^k}
(0,0,0)\tau^{2j+k}+\sum_{i=1}^{\infty}\sum_{j=0}^{\infty}\sum_{k=0}^{\infty}
\frac{\beta^j}{i!j!k!}\cdot
\frac{\partial^{i+j+k}G}{\partial\alpha^i\partial h^j\partial\tau^k}
(0,0,0)\alpha^i\tau^{2j+k}.
\end{align*}
From $(S_3)$, the above equation can be rewritten as
$$
G(\alpha,h(\tau),\tau)=\sum_{2j+k\geqslant
d}\frac{\beta^j}{j!k!}\cdot \frac{\partial^{j+k}G}{\partial
h^j\partial\tau^k}
(0,0,0)\tau^{2j+k}+\sum_{i=1}^{\infty}\sum_{2j+k\geqslant d-m}
\frac{\beta^j}{i!j!k!}\cdot
\frac{\partial^{i+j+k}G}{\partial\alpha^i\partial h^j\partial\tau^k}
(0,0,0)\alpha^i\tau^{2j+k}.
$$

In order to find a solution $\alpha^*=\alpha^*(h(\tau),\tau)$
in the form of $\alpha^*(h(\tau),\tau)=\eta(\tau)\tau^m$ with
$\eta(\tau)$ being a real-valued function of $\tau$, we consider the change of 
variable
$\alpha=\eta \tau^m$, 
\begin{align*}
&G(\alpha,h(\tau),\tau)\\
&=\sum_{2j+k\geqslant
d}\frac{\beta^j}{j!k!}\cdot \frac{\partial^{j+k}G}{\partial
h^j\partial\tau^k}
(0,0,0)\tau^{2j+k}+\sum_{i=1}^{\infty}\sum_{2j+k\geqslant d-m}
\frac{\beta^j}{i!j!k!}\cdot
\frac{\partial^{i+j+k}G}{\partial\alpha^i\partial h^j\partial\tau^k}
(0,0,0)\eta^i\tau^{mi+2j+k}\\
&=\sum_{2j+k=d}\frac{\beta^j}{j!k!}\cdot
\frac{\partial^{j+k}G}{\partial h^j\partial\tau^k} (0,0,0)\tau^{d}+
\mathcal{O}(\tau^{d+1})+\sum_{2j+k = d-m} \frac{\beta^j}{j!k!}\cdot
\frac{\partial^{1+j+k}G}{\partial\alpha\partial h^j\partial\tau^k}
(0,0,0)\eta\tau^{d}+ \mathcal{O}(\tau^{d+1})\\
&=\tau^d\left[\sum_{2j+k=d}\frac{\beta^j}{j!k!}\cdot
\frac{\partial^{j+k}G}{\partial h^j\partial\tau^k}
(0,0,0)+\sum_{2j+k = d-m} \frac{\beta^j}{j!k!}\cdot
\frac{\partial^{1+j+k}G}{\partial\alpha\partial h^j\partial\tau^k}
(0,0,0)\eta + \mathcal{O}(\tau)\right ].
\end{align*}

Denoting by $\tilde G(\eta, \tau)$ the formula in the above bracket.
If $\tau \neq 0$, then $G(\alpha, h(\tau),\tau)=0$ if and only if
$\tilde G(\eta,\tau)=0$. Therefore, we apply the implicit
function theorem to $\tilde G(\eta,\tau)=0$. By $(S_3)$, both $\sum_{2j+k=d}\frac{\beta^j}{j!k!}\cdot
\frac{\partial^{j+k}G}{\partial h^j\partial\tau^k}(0,0,0)$ and
$\sum_{2j+k=d-m}\frac{\beta^j}{j!k!}\cdot
\frac{\partial^{1+j+k}G}{\partial \alpha \partial
h^j\partial\tau^k}(0,0,0)$ are not equal to zero. Let
$$
\eta_0=-\frac{\sum_{2j+k=d}\frac{\beta^j}{j!k!}\cdot
\frac{\partial^{j+k}G}{\partial
h^j\partial\tau^k}(0,0,0)}{\sum_{2j+k=d-m}\frac{\beta^j}
{j!k!}\cdot\frac{\partial^{1+j+k}G}{\partial \alpha \partial
h^j\partial\tau^k}(0,0,0)}, \quad\tau_0=0,
$$
then, $\tilde G(\eta_0,\tau_0)=0$. The functions $\tilde
G(\eta,\tau),$ $\tilde G_{\eta}(\eta, \tau)$ and $\tilde
G_{\tau}(\eta, \tau)$ are continuous in the neighbourhood of the point
$(\eta_0, \tau_0)$, moreover,
$$
\tilde G_{\eta}(\eta_0,
\tau_0)=\sum_{2j+k=d-m}\frac{\beta^j}{j!k!}\cdot
\frac{\partial^{1+j+k}G}{\partial \alpha \partial
h^j\partial\tau^k}(0,0,0)\neq 0.
$$
Hence the implicit function theorem ensures the existence of a
function $\eta=\eta(\tau)$ such that $\tilde G(\eta(\tau),\tau)=0.$

From the equation
$$
\tilde G(\eta,\tau)=\sum_{2j+k=d}\frac{\beta^j}{j!k!}\frac{\partial
^{j+k}G}{\partial h^j \partial
\tau^k}(0,0,0)+\sum_{2j+k=d-m}\frac{\beta^j}{j!k!}\frac{\partial
^{1+j+k}G}{\partial \alpha\partial h^j \partial
\tau^k}(0,0,0)\eta+\mathcal{O}(\tau)=0,
$$
the solution of $G(\alpha, \tau)=0$  takes
the form
$$
\alpha^*(h(\tau),\tau)=\eta(\tau)\tau^m=-\frac{\sum_{2j+k=d-m}\frac{\beta^j}{j!k!}\frac{\partial
^{j+k}G}{\partial h^j \partial
\tau^k}(0,0,0)}{\sum_{2j+k=d-m}\frac{\beta^j}{j!k!}\frac{\partial
^{1+j+k}G}{\partial \alpha \partial h^j \partial
\tau^k}(0,0,0)}\tau^m+\mathcal{O}(\tau^{m+1}),
$$
which completes the proof. 
\end{proof}

From Lemmas \ref{lemma4.1} and \ref{lemma4.2}, we obtain the main result.
\begin{theo}\label{thm4.1}
Let assumptions ($S_1$), ($S_2$) and ($S_3$) be satisfied and let
RK matrix $\tilde A$ be nonsingular. Then the system
(\ref{equation:4.6}) is solved uniquely, i.e., there exists a
unique solution $(Y^T$, $\alpha^*)$.
\end{theo}

 The assumption ($S_1$) and nonsingularity of $\tilde{A}$ in Theorem \ref{thm4.1} can be replaced by $\tau\leqslant h
^2$ and nonsingularity of $A$, respectively. In fact, from \cite{hsy} if RK matrix $A$ is nonsingular, the conclusion of Lemma \ref{lemma4.1} still holds. And  in the proof of Lemma
\ref{lemma4.2}, by considering the change of variable $\alpha(h)=\zeta(h)h^m$, we can still get the solvability of the second equation of \eqref{equation:4.6} respect to
$\alpha$.
The following corollary is a direct result.

\begin{coro}
Let  assumptions $(S_2)$ and $(S_3)$ be satisfied. Then two kinds
of conditions on step sizes and nonsigularties, each of which guarantees the solvability of the
system (\ref{equation:4.6}), are as follows:
\begin{itemize}
\item[(1)] $h\leqslant \tau ^2$ and $\tilde A$ is nonsingular,

\item[(2)]  $\tau \leqslant h^2$ and  $A$ is nonsingular.
\end{itemize}
\end{coro}

\section{Numerical experiments}
\label{sec:5}

In this section, we present some numerical experiments to show  the effectiveness in energy
preserving of the proposed multi-symplectic $\alpha$-RK 
methods.  We consider the sine-Gordon equation (i.e., $V(u)=-\cos(u)$ in \eqref{equation:2.1}):
\begin{equation}\label{sg eq}
\begin{split}
\partial_{tt}u=\partial_{xx}u-\sin(u), \quad &(x,t)\in (-L/2,L/2)\times(0,T],\\
u(-L/2, t)=u(L/2,t), \quad& t\in[0,T],
\end{split}
\end{equation}
with two different  initial conditions: 
$$
u(x,0)=4\tan^{-1}(\frac{e^{x-L/6}}{\sqrt{1-\beta^2}})+4\tan^{-1}(\frac{e^{-x-L/6}}{\sqrt{1-\beta^2}})
$$
and
$$
u(x,0)=\frac{\partial}{\partial t}\left
[4\tan^{-1}(\frac{e^{x-L/6-\beta
t}}{\sqrt{1-\beta^2}})+4\tan^{-1}(\frac{e^{-x-L/6-\beta
t}}{\sqrt{1-\beta^2}})\right ]\Bigg|_{t=0},
$$
respectively.
On an infinite domain, these initial conditions correspond to a
soliton and anti-soliton solution moving with speed $\pm \beta$.
 We set $\beta=0.5$, $L=100$ and $T=200$.

 \subsection{Numerical method}

The numerical solution of \eqref{sg eq} is obtained by using the multi-symplectic
$\alpha\text{-}$Gauss collocation method with $r=1$ (midpoint in
time) and $s=2$ (fourth order Gauss collocation method in space)
with a  temporal step size $\Delta t=0.1$ and a spatial step size $\Delta x=1$ (see Example \ref{example3.1}).

First, in each domain composed of $\Delta t$ and $\Delta x$,
for  multi-symplectic $\alpha \text{-}$Gauss collocation method with
$r=1$ and $s=2$, we obtain a nonlinear system with the discrete
 ECL \eqref{equaion:2.12}-\eqref{equaion:2.13}:
\begin{figure}
\begin{center}
\includegraphics[width=8cm]{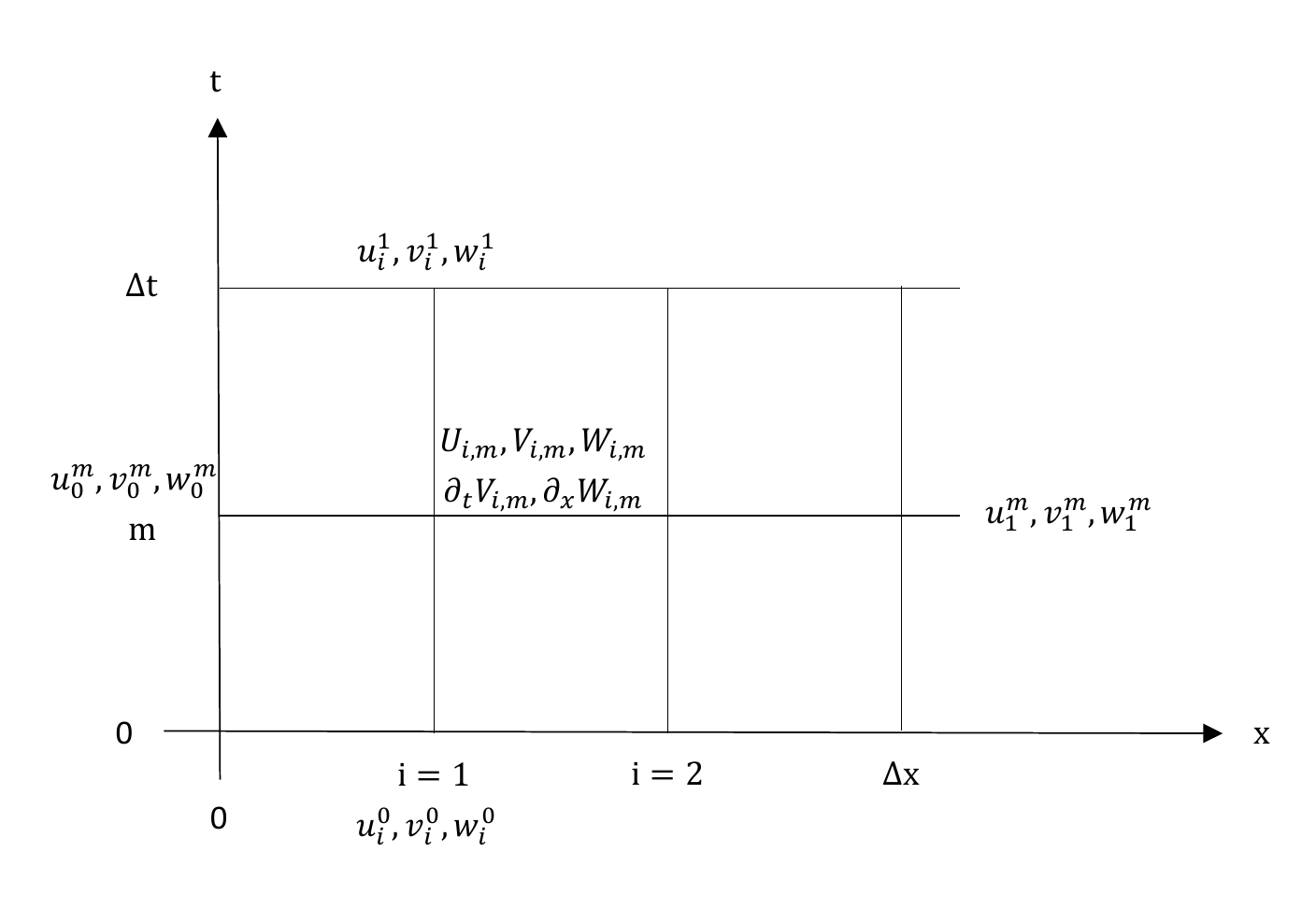}
 \caption{A uniform grid and the unknowns in grid points.} \label{Fig1}
 \end{center}
\end{figure}
\begin{align}
 & U_{1,m}=u^m_0+\Delta x (a_{11}W_{1,m}+(a_{12}-\alpha)W_{2,m}),\nonumber \\
 & U_{2,m}=u^m_0+\Delta x ((a_{21}+\alpha)W_{1,m}+a_{22} W_{2,m}),\nonumber \\
 & W_{1,m}=w^m_0+\Delta x (a_{11}\partial _x W_{1,m}+ (a_{12}- \alpha)
\partial_x  W_{2,m}),\nonumber \\
 & W_{2,m}=w^m_0+\Delta x ((a_{21}+\alpha)\partial _x W_{1,m}+a_{22}
\partial _x W_{2,m}),\nonumber \\
 & u^m_1=u^m_0+\Delta x (b_1 W_{1,m}+b_2 W_{2,m}),\nonumber \\
 & w^m_1=w^m_0+\Delta x (b_1 \partial_x W_{1,m}+b_2 \partial_x
   W_{2,m}),\nonumber \\
 & U_{1,m}=u^0_1+\frac{\Delta t}{2} V_{1,m},\nonumber \\
 & U_{2,m}=u^0_2+\frac{\Delta t}{2} V_{2,m},\nonumber \\
 & V_{1,m}=v^0_1+\frac{\Delta t}{2} \partial_t V_{1,m},   \label{equation:5.1}\\
 & V_{2,m}=v^0_2+\frac{\Delta t}{2} \partial_t V_{2,m},\nonumber \\
 & u^1_1=u^0_1+\Delta t V_{1,m},\nonumber \\
 & u^1_2=u^0_2+\Delta t V_{2,m},\nonumber \\
 & v^1_1=v^0_1+\Delta t \partial_t V_{1,m},\nonumber \\
 & v^1_2=v^0_2+\Delta t \partial_t V_{2,m},\nonumber \\
 & \partial_t V_{1,m}-\partial_x W_{1,m}=-\sin(U_{1,m}),\nonumber \\
 & \partial_t V_{2,m}-\partial_x W_{1,m}=-\sin(U_{2,m}),\nonumber \\
 & \frac{\Delta
 x}{4}\Big((w^1_1)^2+(w^1_2)^2+(v^1_1)^2+(v^1_2)^2-2\cos(u^1_1)-2\cos(u^1_2)\Big)-\Delta t v^m_1
   w^m_1 \nonumber\\
   &=\frac{\Delta
 x}{4}\Big((w^0_1)^2+(w^0_2)^2
+(v^0_1)^2+(v^0_2)^2-2\cos(u^0_1)-2\cos(u^0_2)\Big)-\Delta t v^m_0 w^m_0,\nonumber 
\end{align}
which has 20 unknowns $U_{1,m},
U_{2,m}, V_{1,m}, V_{2,m}, W_{1,m}, W_{2,m},
\partial_t V_{1,m},$ $ \partial_t V_{2,m},$ $\partial_x W_{1,m}, \partial_x W_{2,m},$
$u^m_0, w^m_0,v^m_0, u^1_1, u^1_2, v^1_1, v^1_2, w^1_1, w^1_2$ and
$\alpha$. 
Since \eqref{equation:5.1} only contains $17$ equations, it still needs three more
equations.
 Considering this fact, we add to the above system
the following equations,
\begin{equation}\label{equation:5.3}
\begin{aligned}
& u^m_0=u^0_0+\Delta t \tilde{a}_{11} v^m_0, \\
& u^1_1=u^1_0+\Delta x (a_{11} w^1_1+a_{12} w^1_2),\\
& u^1_2=u^1_0+\Delta x (a_{21} w^1_1+a_{22} w^1_2),\\
&u^1_0=u^0_0+\Delta t \tilde{b}_1 v^m_0,
\end{aligned}
\end{equation}
where $u^1_0$ is another unknown.
Therefore, we have the nonlinear system
(\ref{equation:5.1})-(\ref{equation:5.3}) with 21 unknowns.

\begin{figure}[ht]
\begin{center}
\includegraphics[width=7cm,height=4.5cm]{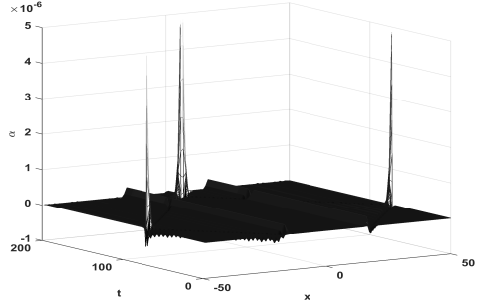}
\caption{Values of $\alpha$ in $\alpha$-Gauss collocation methods at $(x,t)$-plane
[-50, 50]$\times$ [0, 200] with $h=1$ and $\tau=0.1$.} \label{Fig2}
\end{center}
\end{figure}

\begin{figure}
\begin{center}
\includegraphics[width=8cm]{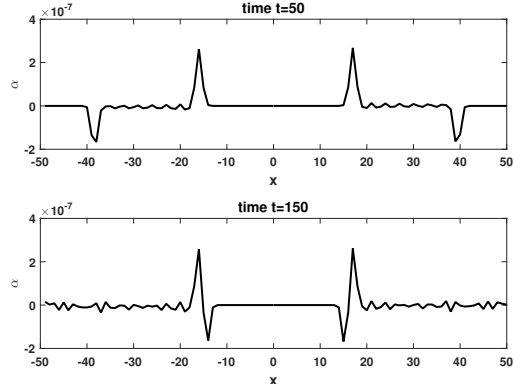}\\
\caption{Values of $\alpha$ in $\alpha$-Gauss collocation methods at two fixed moments t=50 and t=150.}\label{Fig3}
\end{center}
\end{figure}

Next, considering the periodic boundary condition, in each time level we
put together the above individual nonlinear systems through space
axis (including $M=100$ grid points), which leads to a nonlinear
system with $2100$ unknowns. Denoting by $X$  the unknowns, this nonlinear system can be rewritten as
$
F(X)=0,
$
 which is solved by using Newton iteration method with tolerance
$\epsilon=10^{-15}$.

\subsection{Numerical results}

Fig.~\ref{Fig2} presents the values of parameter $\alpha$ in multi-symplectic
$\alpha$-Gauss collocation method on $(x,t)$-plane [-50,
50]$\times$ [0, 200], which make the method 
preserve energy.
Fig.~\ref{Fig3} shows the values of  such parameter in details, for selected moments $t=50$ and $t=150$,
at all spatial grid points.
\begin{figure}
\begin{center}
\includegraphics[width=8cm]{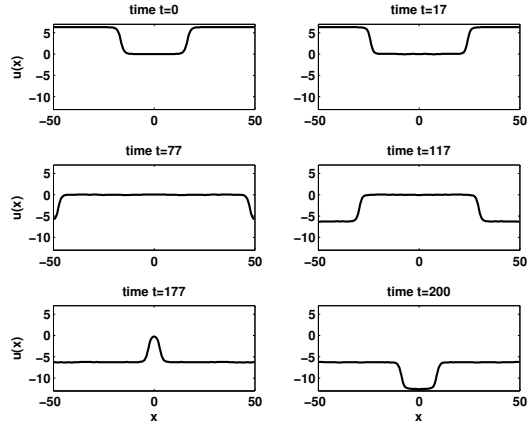}
 \caption{Time evolution of solution during the time interval t $\in [0, 200]$.}
 \label{Fig4}
 \end{center}
\end{figure}
 Numerical results indicate that such parameter sequence exists at every
calculation grids and has the absolute values closed to zero (about
$10^{-7} \thicksim 10^{-8}$).
Fig.~\ref{Fig4} is the cross-section plots  of the wave $u(x,t)$
 at some selected moments during the time interval $[0,200]$.

We compare the errors of the total energy ${\cal E}$$(t)$ by using
MSRK (Gauss Collocation, \cite{r}) and
$\alpha$-MSRK  ($\alpha\text{-}$Gauss Collocation) methods, respectively, which are shown in  Fig.~\ref{Fig5}.
 Numerical results show that in the
conservation of the total energy ${\cal E}$$(t)$, the error of
MSRK method is of about $10^{-3}$, while the error of $\alpha$-MSRK
method is of about $10^{-12}$. Both MSRK method
and $\alpha$-MSRK method conserve the total momentum ${\cal I}$$(t)$ exactly, since it is a
quadratic invariant.

\begin{figure}[tp!]
\begin{center}
\includegraphics[width=8cm]{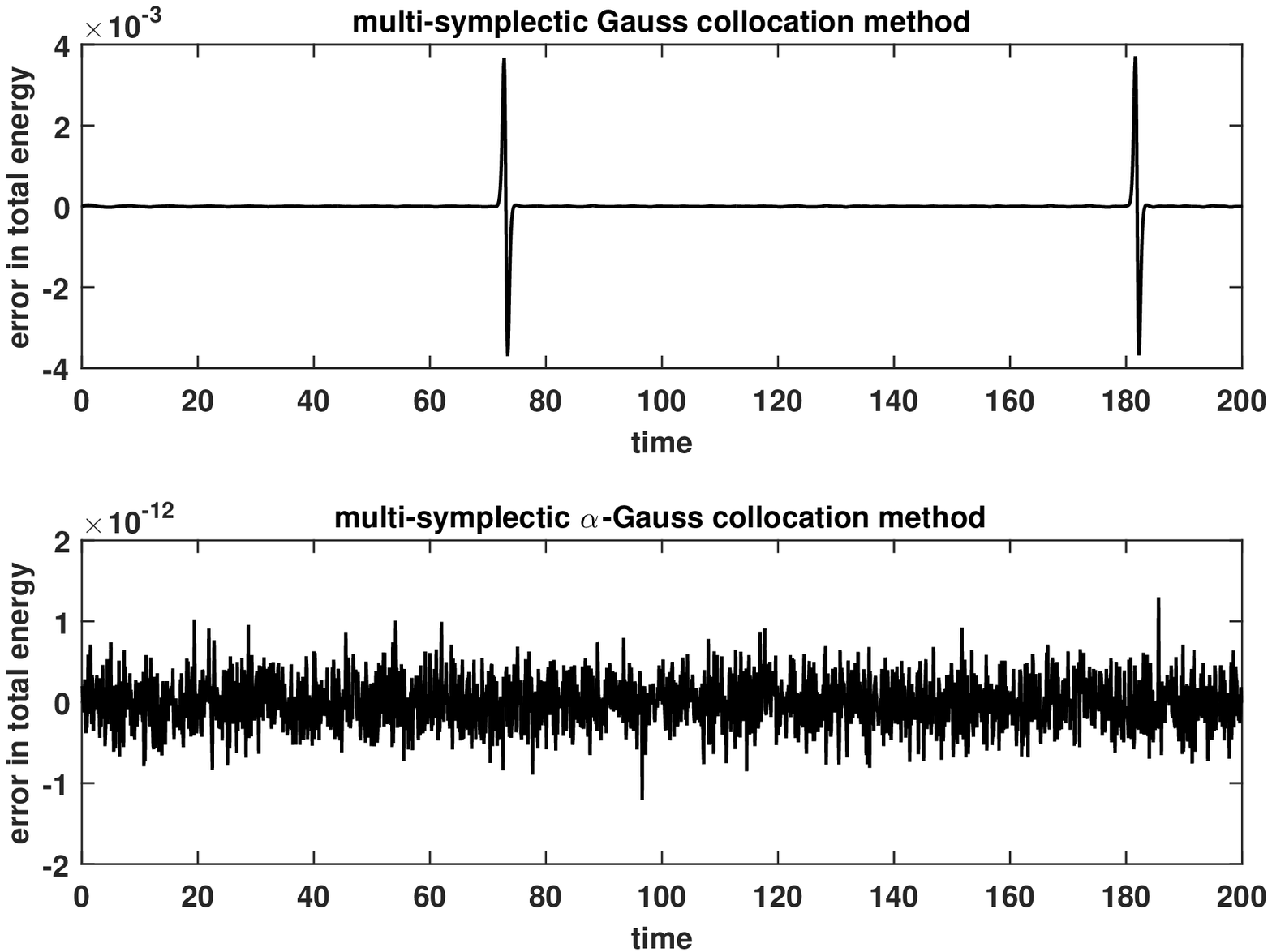}
 \caption{Comparison of numerical errors in the global energy for both multi-symplectic Gauss collocation method and
multi-symplectic $\alpha\text{-}$Gauss collocation method over the
time interval $[0, 200]$.}\label{Fig5}\end{center}
\end{figure}

Fig.~\ref{Fig6} shows the comparison of the error in the local discrete ECL \eqref{equaion:2.12}-\eqref{equaion:2.13} of MSRK method and
 $\alpha$-MSRK method, 
 in the time intervals $[0,30]$ and $[150, 170]$, respectively.  
Observe that $\alpha$-MSRK  method (about $10^{-13}$) conserves
 the  discrete ECL  better than MSRK method (about
$10^{-3}$).

\begin{figure}[tp!]
\begin{center}
\includegraphics[width=8cm]{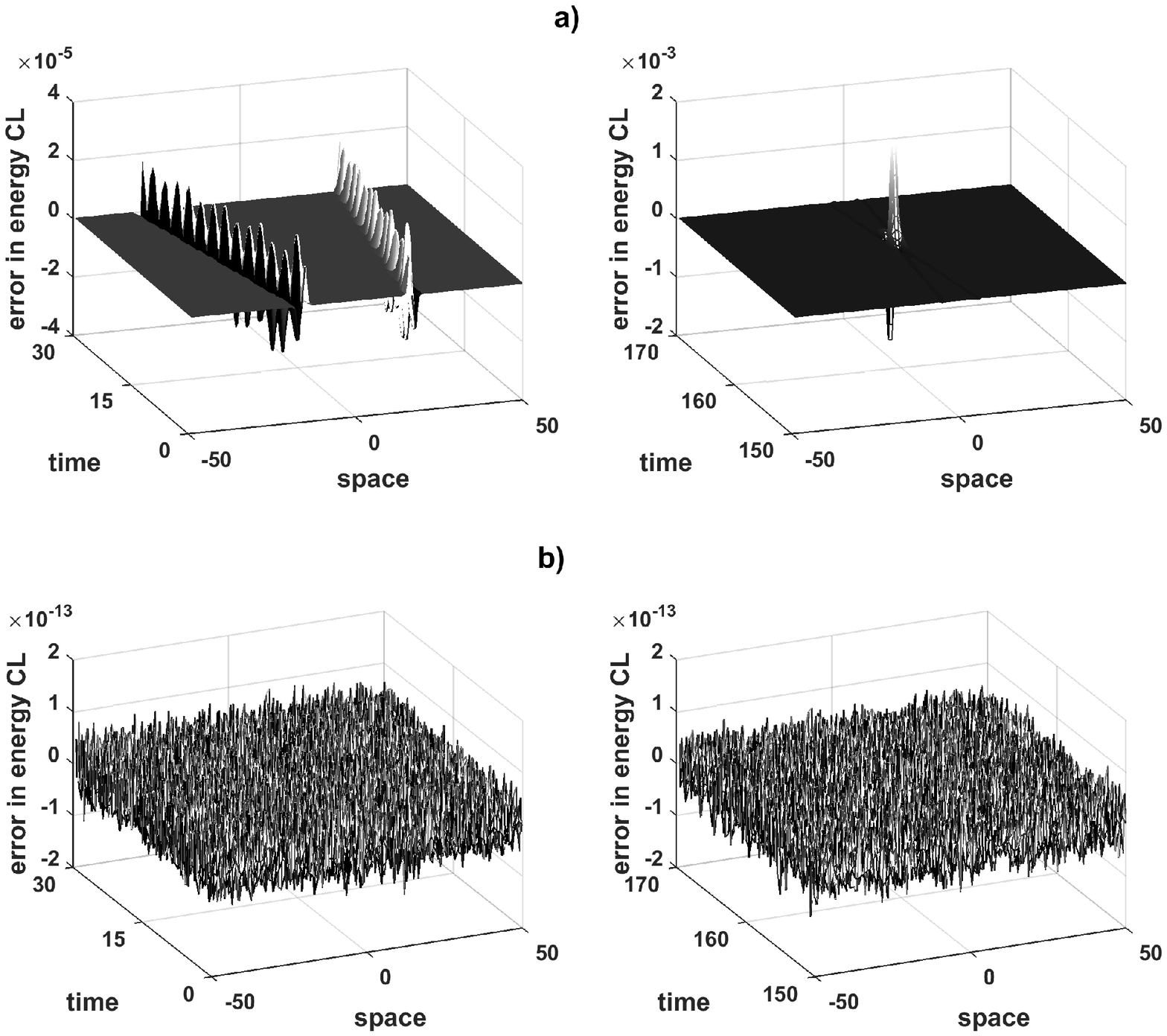}\\
 \caption{Comparison of numerical errors in the local energy for both multi-symplectic Gauss collocation method and
multi-symplectic $\alpha$-Gauss collocation method over the time
intervals $[0, 30]$ and $[150, 170]$.}\label{Fig6}

\end{center}
\end{figure}

\subsection{Comparison of computational costs}

When we implement the numerical experiments to solve the problem with periodic
boundary condition by using the multi-symplectic
$\alpha\text{-}$RK method, the major costs result from the discretization of PDEs. The
nonlinear system to be solved in each domain $(0, \Delta
x)\times (0, \Delta t)$ has $(5s+3)r+3s+2$ unknowns. When $\alpha=0$, namely, the standard multi-symplctic
 RK method is applied, the nonlinear system has $(5s+2)r$
unknowns. Therefore when Newton
iteration method is applied, Jacobian matrix is of
$(5s+2)r$  dimension for  the case of $\alpha =0$, while  it is of
$(5s+3)r+3s+2$ dimension for the case of $\alpha=\alpha^*$. After assembling individual Jacobian
matrices through space axis (including $M=100$ grid points), we obtain the matrices with $M(5s+2)r$
and $M((5s+3)r+3s+2)$ dimensions respectively. These matrices have approximately sparse band structure with
bandwidth $(5s+2)r+2$ and $(5s+3)r+3s+4$, respectively. If we
denote the bandwidth by $d$ and dimension of band-matrix by $N$,
the calculation cost of Gauss elimination method can be written as follows:
\begin{align*}
{\rm Cost}&=[(N-d)d^2+d^3/3]({\rm forward})+(Nd-d^2/2)({\rm backward})\\
&=Nd^2-2d^3/3+Nd-d^2/2.
\end{align*}

If $N\gg d$, then ${\rm Cost}\approx Nd^2+Nd\approx Nd^2$.
Therefore, the costs of these two methods are approximately $M((5s+2)r+2)^2$ and
$M((5s+3)r+2s+4)^2 $, respectively.
From this point, we know that the larger $s$ and $r$  are, the smaller
the difference in dimensions of individual Jacobian matrices
is. 
We compare the computational time needed to solve the sine-Gordon
equation \eqref{sg eq} by using the multi-symplectic $\alpha$-Gauss collocation
method with $r=1$ and $s=2$  in the cases of $\alpha=0$ and
$\alpha=\alpha^{*}$. In the numerical experiments, we have actually used Gauss
elimination method to solve the linear system in each step of Newton
iterations. We notice from Table~\ref{Table1}  that  the  costs in numerical experiments coincide
approximately with the above theoretical results.

\begin{table}
\begin{center}
  \begin{tabular}{|c|c|c|}
    \hline
    
    \hline
      ${\rm Computer}$ ${\rm Type}$ &$\alpha=0$& $\alpha=\alpha^{*}$\\
    \hline
    ${\rm Core}$ ${\rm Two}$ ${\rm Duo}$ & $1091.47s$ \quad   & $3283.46s$ \quad \\
    \hline
     ${\rm Four}$ ${\rm Generation}$ ${\rm Co}$ ${\rm i5}$ & $482.3s$ \quad   & $1320.8s$ \quad  \\
    \hline
    
    \hline
  \end{tabular}\\
\caption{Computer costs for the cases of $\alpha=0$ and $\alpha=\alpha^*$.}\label{Table1}
\end{center}
\end{table}

\section{Concluding remarks}
\label{sec:6}

We propose a family of multi-symplectic  $\alpha$-RK
methods for Hamiltonian wave equations. These methods are
multi-symplectic perturbations of the classical multi-symplectic RK
methods with the free parameter $\alpha$. The existence of a parameter such that the methods are energy-preserving in a weaker sense is proved. This weaker sense means that the parameter depends on the step sizes and initial data, which says that the energy conservation property may fail if one changes the step sizes or the initial data.  
 Numerical
experiments show the effectiveness of the proposed method, and
the preservation of both multi-symplecticity and energy.  
This work considers only energy-preserving property
 together with multi-symplecticity, 
a research on multi-symplectic
RK method preserving both energy and momentum, further multiple
invariants,  is the future work.
\section*{Acknowledgements}
Jialin Hong and Chuchu Chen are supported by National Natural Science Foundation of China (NOs.  91630312, 11711530017, and 11871068). Chol Sim and  Kwang Sonwu are supported
by Academy of Mathematics and Systems Science, Chinese Academy of Sciences.



\end{document}